\documentclass[12pt]{amsart}
\usepackage{fullpage,url,amssymb,enumerate}
\usepackage[all,cmtip]{xy} % for complicated commutative diagrams
\usepackage{mathrsfs} % for \mathscr (script letters)
\usepackage[alphabetic,lite]{amsrefs} % for bibliography

\DeclareFontEncoding{OT2}{}{} % to enable usage of cyrillic fonts
\newcommand{\textcyr}[1]{%
 {\fontencoding{OT2}\fontfamily{wncyr}\fontseries{m}\fontshape{n}\selectfont #1}}
\newcommand{\Sha}{{\mbox{\textcyr{Sh}}}}

\DeclareMathOperator{\Cl}{Cl}
\DeclareMathOperator{\disc}{disc}
\DeclareMathOperator{\ord}{ord}
\DeclareMathOperator{\rk}{rk}
\DeclareMathOperator{\im}{im}

\DeclareMathOperator{\Aut}{Aut}
\DeclareMathOperator{\Norm}{Norm}

\newcommand{\Selfake}{\mathrm{Sel}_\mathrm{fake}}
\renewcommand{\div}{\mathrm{div}}
\newcommand{\nd}{\mathrm{nd}}

\newcommand{\FF}{\mathbb{F}}	
\newcommand{\QQ}{\mathbb{Q}}
\newcommand{\RR}{\mathbb{R}}
\newcommand{\ZZ}{\mathbb{Z}}

\newcommand{\calO}{\mathcal{O}}

\newcommand{\fa}{\mathfrak{a}}
\newcommand{\fp}{\mathfrak{p}}
\newcommand{\fq}{\mathfrak{q}}
\newcommand{\ft}{\mathfrak{t}}

\newcommand{\legendre}[2]{\left(\frac{#1}{#2}\right)}
\newcommand{\tleg}[2]{\big(\frac{#1}{#2}\big)}
\newcommand{\act}[2]{{}^{#1}\!#2}

\newtheorem{theorem}{Theorem}[section]
\newtheorem*{theoremA}{Theorem A}
\newtheorem*{theoremB}{Theorem B}
\newtheorem{lemma}[theorem]{Lemma}
\newtheorem{corollary}[theorem]{Corollary}
\newtheorem{proposition}[theorem]{Proposition}

\theoremstyle{definition}
\newtheorem{definition}[theorem]{Definition}

\newtheorem*{questionA}{Question A}
\newtheorem*{questionB}{Question B}

\newtheorem{assumption}[theorem]{Assumption}

\theoremstyle{remark}
\newtheorem{remark}[theorem]{Remark}

\makeatletter
\newcommand{\customlabel}[2]{%
\protected@write \@auxout {}{\string \newlabel {#1}{{#2}{}}}}
\makeatother

\begin{document}
\title{On congruent primes and class numbers of imaginary quadratic fields}
%%%% ADD AUTHORS
\author{Nils Bruin}
\thanks{Research of first author supported by NSERC}
\address{Department of Mathematics, Simon Fraser University,
         Burnaby, BC V5A 1S6, Canada}
\email{nbruin@sfu.ca}
\urladdr{http://www.cecm.sfu.ca/\~{}nbruin}

\author{Brett Hemenway}
\address{Department of Mathematics, University of Michigan,
         Ann Arbor, MI 48109-1043, United States of America}

\date{April 27, 2012}

\begin{abstract}
We consider the problem of determining whether a given prime $p$ is a congruent number. We present an easily computed criterion that allows us to conclude that certain primes for which congruency was previously undecided, are in fact not congruent. As a result, we get additional information on the possible sizes of Tate-Shafarevich groups of the associated elliptic curves.

We also present a related criterion for primes $p$ such that $16$ divides the class number of the imaginary quadratic field $\QQ(\sqrt{-p})$. Both results are based on descent methods.

While we cannot show for either criterion individually that there are infinitely many primes that satisfy it nor that there are infinitely many that do not, we do exploit a slight difference between the two to conclude that at least one of the criteria is satisfied by infinitely many primes.
\end{abstract}

\maketitle

\section{Introduction}

The results in this article are inspired by two related questions.
\begin{enumerate}
 \item What exponents can occur for class groups of number fields?
 \item What exponents can occur for Tate-Shafarevich groups of abelian varieties?
\end{enumerate}
In particular, we consider the $2$-power part of class groups of imaginary quadratic fields and the $2$-power part of the Tate-Shafarevich groups of quadratic twists of the elliptic curve $E_1\colon y^2=x^3-x$; these are the curves that play a role in the classical \emph{congruent number problem}.

In either case, it is known that the \emph{size} of the $2$-power parts of the groups can be made arbitrarily large by making the $2$-torsion subgroups arbitrarily large. For class groups, these results come from Gauss's genus theory (see \cite{lemmermeyer:reciprocity}*{2.2}) and for Tate-Shafarevich groups from an analogous construction (see \cite{kramer:large_sha}).

We limit ourselves to imaginary quadratic fields $\QQ(\sqrt{-p})$, and quadratic twists $E_p\colon y^2=x^3-p^2x$, where in either case $p$ is a prime. It is known that there are infinitely many primes $p$ such that the group classes of fractional ideals modulo principal ideal $\Cl(\QQ(\sqrt{-p}))$ contains elements of order $2$, $4$, or $8$. Existing results establish a similar fact for for the Tate-Shafarevich group $\Sha(E_p)$ of $E_p$, namely that there are infinitely many primes $p$ such that $\Sha(E_p)[2]\simeq(\ZZ/2)^2$ and, if one assumes that elliptic curves of rank $2$ are extremely rare, that one has $\ZZ/4\hookrightarrow\Sha(E_p)$ for infinitely many $p$.
The proofs in either case consist of observing that the answer to either question is governed by the splitting of $p$ in some fixed number field. The Chebotarev Density Theorem then guarantees the existence of infinitely many $p$ with the desired property.

Our results provide a characterization of the primes for which the relevant groups are one step bigger. Unfortunately, the result does not seem to correspond to a splitting condition in some fixed extension anymore, so an infinite number of primes satisfying the criterion is not guaranteed. We introduce some notation to formulate our results precisely.

It follows from genus theory that for $p\equiv 3\pmod{4}$, the class number of $\QQ(\sqrt{-p})$ is odd. For $p\equiv 1\pmod{4}$, we write $h(-4p)=\#\Cl(\ZZ[\sqrt{-p}])$, which is equal to the class number of $\QQ(\sqrt{-p})$.
In this case genus theory also guarantees that $\Cl(\ZZ[\sqrt{-p}])/2\Cl(\ZZ[\sqrt{-p}])\simeq \ZZ/2$, so the $2$-power part of $\Cl(\ZZ[\sqrt{-p}])$ is cyclic.

\begin{questionA}\customlabel{Q:classnr}{A}
Given $e\geq0$, how can we characterize the primes $p$ such that $2^e\mid h(-4p)$?
\end{questionA}

Note that $2$ is ramified in $\QQ(\sqrt{-p})$ if $p\equiv 1\pmod 4$, so there is an ideal $\ft\subset\ZZ[\sqrt{-p}]$ such that $\ft^2=2\ZZ[\sqrt{-p}]$.
Since $\ZZ[\sqrt{-p}]$ does not contain an element of norm $2$, 
we have that $[\ft]\in\Cl(\ZZ(\sqrt{-p}))$ is of order $2$. Answers to Question~\ref{Q:classnr} can therefore take the form of descriptions of the sets
\[
\begin{aligned}
V(e)&=\{p\text{ prime}: p\equiv 1\pmod{4} \text{ and } 2^e\mid h(-4p)\},\\
    &=\{p\text{ prime}: p\equiv 1\pmod{4} \text{ and } [\ft]\in 2^{e-1}\Cl(\QQ(\sqrt{-p}))\}.
\end{aligned}
\]
Classical results together with Barrucand-Cohn (see \cite{barcoh:x2p32y2}) establish (see Section~\ref{S:notation} for notation)
\begin{subequations}
\begin{align}
\label{E:V(1)}V(1)&=\{p\text{ prime}: p \equiv 1\pmod{4}\},\\
\label{E:V(2)}V(2)&=\{p\text{ prime}: p \equiv 1\pmod{8}\},\\
\label{E:V(3)}V(3)&=\{p\text{ prime}: p \equiv 1\pmod{8} \text{ and }\legendre{1+i}{p}=1\}.
\end{align}
The set $V(3)$ consists exactly of the primes that completely split in $K_1=\QQ(\sqrt{1+i})=\QQ(\alpha)$, where $\alpha^4-2\alpha^2+2=0$. Let $\delta_p\in K_1$ be an algebraic integer satisfying $\Norm_{K_1/\QQ(i)}(\delta_p)=p$ and let $\fp_p$ be a prime of $K_1$ above $p$ such that $\delta_p\notin\fp$. We will check that the quadratic symbol $\tleg{\alpha\delta_p}{\fp}$ does not depend on the actual choices of $\delta_p$ and $\fp$. We prove
\begin{theoremA}\customlabel{T:V(4)}{A}
\begin{equation}\label{E:V(4)}
 V(4)=\{p\text{ prime}: p \equiv 1\pmod{8} \text{ and }\legendre{1+i}{p}=\legendre{\alpha\delta_p}{\fp_p}=1\}.\\
\end{equation}
\end{theoremA}
\end{subequations}
In \cite{stevenhagen:2powers}*{Theorem~2}, which provides references for \eqref{E:V(1)},\eqref{E:V(2)},\eqref{E:V(3)}, there is a different criterion for membership of $V(4)$, in terms of the $2$-adic logarithm of the fundamental unit of $\QQ(\sqrt{p})$.

For $\Sha(E_p)$ we proceed in a way similar to Question~\ref{Q:classnr}. For an abelian group $G$ we write $G[2^\infty]$ for 
its $2$-primary subgroup.
\begin{questionB}\customlabel{Q:congnum}{B}
Given $e\geq 0$, how can we characterize the primes $p$ such that $\Sha(E_p)[2^\infty]$ contains element of order $2^e$?
\end{questionB}
In classical terminology, a positive integer is called \emph{congruent} if $E_n(\QQ)$ has positive rank. In order to avoid confusion with other uses of the word, we italicize it whenever used with this meaning.

It is already known (see Section~\ref{S:congknown}) that if $n=p$ is a prime with $p\not\equiv 1\pmod{8}$ then $\Sha(E_p)[2^\infty]$ is trivial.
Therefore, we concentrate on the case $p\equiv 1\pmod{8}$. Then either $p$ is \emph{congruent} or $\Sha(E_p)[2]\simeq(\ZZ/2)^2$. In fact, we can write down $3$ principal homogeneous spaces of $E$, given by
\begin{equation}\label{E:Cpi}
\begin{aligned}
C_{p,1}\colon& y^2=p(x^4-6x^2+1)=p(x^2+2x-1)(x^2-2x-1)\\
C_{p,2}\colon& y^2=p(x^4+4)=p(x^2+2x+2)(x^2-2x+2)\\
C_{p,3}\colon& y^2=p(x^4+1),
\end{aligned}
\end{equation}
which have points everywhere locally. They represent possibly trivial classes $\xi_1,\xi_2,\xi_3\in \Sha(E_p)[2]$ which generate the group and satisfy $\xi_1+\xi_2=\xi_3$. To test the triviality of $\xi_i$, we are led to considering
\[
\begin{aligned}
W(e)=\{p\text{ prime}: p\equiv 1\pmod{8} \text{ and } \xi_i\in 2^{e-1}\Sha(E_p)\}.
\end{aligned}
\]
The \emph{Cassels-Tate pairing} \cite{cassels:arith_g1-IV} implies that if there exists an $e\geq 1$ such that
$\xi_i\notin 2^{e}\Sha(E_p)$ (i.e., $\xi_i$ is not totally $2$-divisible) then $\Sha(E_p)[2]\simeq(\ZZ/2)^2$. In that case no $C_{p,i}$ has a rational point and $p$ is not \emph{congruent}.  also implies that the definition of $W(e)$ does not depend on which $\xi_i$ is chosen.

For Question~\ref{Q:classnr}, we know that $[\ft]$ is nontrivial, so divisibility of this class directly yields results on $h(-4p)$. We do not have a corresponding guarantee for Question~\ref{Q:congnum}. For instance, if all $\xi_i$ are trivial, as happens for $p=41$, then $p\in W(e)$ for all $e$ but $\Sha(E_p)[2^\infty]=0$. On the other hand, if we establish that $p\in W(e)$ and $p\notin W(e+1)$ then it does follow that $\xi_i$ is nontrivial and we can conclude that $2^e$ divides  the exponent of $\Sha(E_p)$.

Results attributed to A.~ Genocchi and L.~Bastien (see \cite{dickson:histvolII}*{Chapter~XVI} and \cite{tunnell:congnum}) essentially establish
\begin{subequations}
\begin{align}
\label{E:W(1)} W(1)&=\{p\text{ prime}: p \equiv 1\pmod{8}\},\\
\label{E:W(2)} W(2)&=\{p\text{ prime}: p \equiv 1\pmod{8} \text{ and }\legendre{1+i}{p}=1\}.
\end{align}
In the statement below, we use the same $\delta_p,\fp_p$ as in Theorem~\ref{T:V(4)}. 
Furthermore, let $\zeta$ be a primitive eighth root of unity. For $p\equiv 1\pmod 8$, we have that $\zeta\in\QQ_p$. We prove
\begin{theoremB}\customlabel{T:W(3)}{B}
\begin{equation}\label{E:W(3)}
 W(3)=\{p\text{ prime}: p \equiv 1\pmod{8} \text{ and }\legendre{1+i}{p}=\legendre{\zeta\alpha\delta_p}{\fp_p}=1\}.\\
\end{equation}
\end{theoremB}
\end{subequations}

While the criteria in Theorems~\ref{T:V(4)} and \ref{T:W(3)} do not immediately guarantee that there are infinitely many primes satisfying them, the fact that the descriptions do not completely agree allows us to conclude:
\begin{corollary}
At least one of $W(3)$ and $V(4)$ is infinite.
\end{corollary}
\begin{proof}
Note that $W(2)=V(3)$ contains infinitely many primes $p$ satisfying $p\equiv 9\pmod{16}$. For these primes we have $\tleg{\zeta}{p}=-1$, which is exactly the symbol by which the descriptions of $V(4)$ and $W(3)$ differ. Therefore we have either $p\in V(4)$ or $p\in W(3)$. It follows that at least one set must be infinite. 
\end{proof}

In Section~\ref{S:implications} we derive some more results along these lines. Let us conclude with noting that the criteria for $W(3)$ and $V(4)$ are easy to test computationally for individual primes.

\begin{remark}\label{R:1e200}
It is easy to compute with a computer algebra system that $10^{200}+16737$ is the first prime beyond $10^{200}$ such that $p\in V(3)=W(2)$, but $p\notin V(4), W(3)$ and that $q=10^{200}+28729$ is the first such prime such that $q\in V(4)$ but $q\notin W(3)$. In particular neither prime is \emph{congruent}.
\end{remark}

\section{Implications}\label{S:implications}

Just from equations \eqref{E:W(1)} and \eqref{E:W(2)} it follows that there are infinitely many primes $p\in W(1)\setminus W(2)$. Hence there are infinitely many primes $p$ with $(\ZZ/2\ZZ)^2\hookrightarrow\Sha(E_p)$.

Note that elliptic curves of rank bigger than $1$ seem very rare, so one would expect that for most $p\in W(2)$ it is still the case that at least one $\xi_i$ is non-trivial. Indeed, the discussion in \cite{rubsil:ranks}*{Section~7} suggests that the following is plausible.
\begin{assumption}[Goldfeld for primes]\label{A:goldfeld}
The primes $p$ for which $E_p(\QQ)$ has rank $2$ have asymptotic density $0$ in the set of all primes.
\end{assumption}
With this assumption, Equation~\eqref{E:W(2)} would imply that there are infinitely many $p$ for which $\ZZ/4\hookrightarrow \Sha(E_p)$. If in addition we assume that only the trivial element in $\Sha(E_p)$ is totally divisible, then \cite{dokchitser:BSDmod2}*{Corollary~4.20} implies that the parity conjecture holds for $E_p$. This would exclude the possibility that  $E_p(\QQ)$ has rank $1$ for $p\equiv 1\pmod{8}$ and we obtain that for infinitely many $p$ we have $(\ZZ/4)^2\hookrightarrow \Sha(E_p)$.

Numerical data suggests that $V(4)$ has asymptotic density $\frac{1}{2}$ in $V(3)$. Indeed, it has been conjectured \cites{ cohn-lagarias:density2, cohn-lagarias:density1} that such a density exists, but to our knowledge this conjecture is still open. Comparison of our descriptions of $V(4)$ and $W(3)$ shows that $W(3)$ would have an asymptotic density if and only if $V(4)$ has one. At least one would expect that $V(4)$ and $W(3)$ are both infinite.  We can combine Theorems~\ref{T:V(4)} and \ref{T:W(3)} to prove half of that.
\begin{corollary}
At least one of the following statements is true.
\begin{itemize}
\item[(a)] There are infinitely many primes $p$ such that $(\ZZ/4)^2\hookrightarrow\Sha(E_p)$.
\item[(b)] There are infinitely many primes $p$ such that $16\mid h(-4p)$.
\end{itemize}
\end{corollary}
\begin{proof}
The first statement holds for primes $p\in W(2)\setminus W(3)$, while the second statement holds for primes $p\in V(4)$. The intersections of $V(4)$ and $W(3)$ with $p\equiv 1\pmod{16}$ coincide. If $V(4)$ were finite then there would be only finitely many primes in $W(3)$ that satisfy $p\equiv 1\pmod{16}$. Since $W(2)$ contains infinitely many such primes, the corollary follows.
\end{proof}

Again, using Assumption~\ref{A:goldfeld} we can obtain a stronger, conditional result.
\begin{corollary}
Under Assumption~\ref{A:goldfeld}, at least one of the following statements is true.
\begin{itemize}
\item[(a)] There are infinitely many primes $p$ such that $\ZZ/8\hookrightarrow\Sha(E_p)$.
\item[(b)] There are infinitely many primes $p$ such that $16\mid h(-4p)$.
\end{itemize}
\end{corollary}
\begin{proof}
The first statement follows if $W(3)$ contains a set of positive asymptotic density in the primes and the second follows if $V(4)$ is infinite. When restricted to primes $p\equiv 9\pmod{16}$, the two sets are complementary in $V(3)=W(2)$. Hence if $V(4)$ contains only finitely many primes congruent to $9$ modulo $16$, then $W(3)$ does contain a positive density set. The corollary follows.
\end{proof}

\section{Some related modular results}

Observations going back to Gauss (see \cite{dickson:histvolIII}*{Chapter~VI}) link
class numbers to coefficients of modular forms of weight $\frac{3}{2}$, in particular the cube of the classical \emph{Jacobi $\Theta$-series}. We write
\[ \sum_{n=0}^\infty r(n) q^n = \Theta(q)^3 = \left(\sum_{n=-\infty}^\infty q^{n^2}\right)^3.\]
The class number relation relevant for our problem is that for primes $p\equiv 1\pmod 4$ we have
\[h(-4p)=\frac{r(p)}{12}.\]
The other coefficients relate to class groups as well.
For any particular $p$ one can use this relation, or other methods, to compute $h(-4p)$ and hence decide for which $e$ one has $p\in V(e)$, at least in principle.

For the \emph{congruent number} problem, Tunnell \cite{tunnell:congnum} identified a specific modular form
\[\sum a_n q^n\in S_\frac{3}{2}(\tilde{\Gamma}_0(128))\]
such that for odd $n$ we have that $a_n\neq0$ implies that $n$ is not \emph{congruent}. He also gives another form for even $n$. His result relies on Waldspurger's work on the Shimura Correspondence and the part of the Birch--Swinnerton-Dyer conjecture (BSD) proved by Coates--Wiles. Tunnell also observes that the full BSD-conjecture implies that for non-\emph{congruent} primes $p$ we have
\[\#\Sha(E_p)=\tfrac{1}{4}a_p^2.\]
Rubin's work \cite{rubin:BSD-CM} imposes severe restrictions on the values that $\#\Sha(E_p)/a_p^2$ can take, but it does not provide any information on $\ord_2(\#\Sha(E_p)/a_p^2)$. Therefore, even though the analytic approach does provide means to prove that numbers are not \emph{congruent}, it requires unproven parts of BSD to provide any results for
Question~\ref{Q:congnum}.

It is also worthwhile to note that for neither question would analytic approaches be feasible to answer questions for primes in the range of Remark~\ref{R:1e200}. This is not too surprising, since the analytic approaches would find the integers $h(-4p)$ (unconditionally) and $\#\Sha(E_p)$ (conditionally), whereas Theorems~\ref{T:V(4)} and \ref{T:W(3)} only provide information on the valuation at $2$ of those integers.

\section{Preliminaries}
\label{S:notation}

When $K$ is a number field, we write $\calO_K$ for its ring of integers and $\calO_K^\times$ for its group of units. We write $\Cl(K)=\Cl(\calO_K)$ for its ideal class group. When $S$ is a finite set of places of $K$, we write $\calO_{K,S}$ for the ring of $S$-integers.

If $\fp\subset\calO_K$ is a prime ideal, we write $\tleg{.}{\fp}\colon \calO_K/\fp\to\{0,\pm 1\}$ for the associated quadratic character on the residue field, extended by setting $\tleg{0}{\fp}=0$.
When $\fp$ is a principal ideal generated by $\pi\in\calO_K$, we write $\tleg{.}{\pi}=\tleg{.}{\fp}$.
For an element $\alpha\in \calO_K$ we write $\tleg{\alpha}{\fp}$ for the quadratic character of the natural image of $\alpha$ in $\calO_K/\fp$.
When $\fp$ is completely split over a rational prime $p$, we denote $\tleg{\alpha}{\fp}=\tleg{\alpha}{p}$ if the value of the symbol is the same for all $\fp$ dividing $p\calO_K$. In this case the symbol can be computed by taking any element $\alpha'\in\FF_p$ that is a root of the minimal polynomial of $\alpha$ modulo $p$ and computing the Legendre symbol $\tleg{\alpha'}{p}$.

In what follows we need a variety of number fields. We fix notation and names for these fields. Let $p$ be a rational prime. We consider the following extensions.
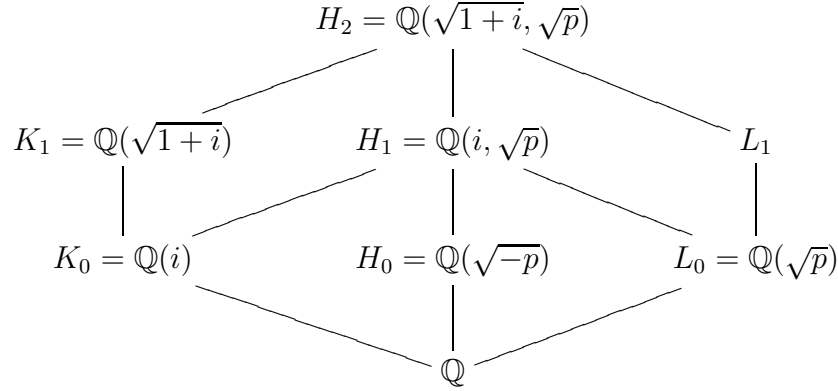
\[\xymatrix{
&H_2=\QQ(\sqrt{1+i},\sqrt{p})\ar@{-}[dl]\ar@{-}[d]\ar@{-}[dr]\\
  K_1=\QQ(\sqrt{1+i})\ar@{-}[d]&
  H_1=\QQ(i,\sqrt{p})\ar@{-}[dl]\ar@{-}[d]\ar@{-}[dr]&
  L_1\ar@{-}[d]\\
K_0=\QQ(i)\ar@{-}[dr]&H_0=\QQ(\sqrt{-p})\ar@{-}[d]&L_0=\QQ(\sqrt{p})\ar@{-}[dl]\\
&\QQ
}\]
If $p\neq 2$ then $L_1$ can be described as the unique quartic subfield of $H_2$ that contains $\sqrt{p(1+i)}$. 

The choice $p=2$ plays a special role. We fix separate names $M_i$ for $H_i$ and $N_0$ for $L_0$. Note that $M_2$ is galois over $\QQ$ and that it contains two conjugate subfields isomorphic to $K_1$. We identify $K_1$ with one of them. We write $N_1$ for one of the non-normal quartic subfields containing $N_0$. 

We will conduct some involved computations in $M_2$ and its subfields. Some of these computations depend on the conjugates chosen. To avoid confusion,  we fix a generator $\beta$ for $M_2$, satisfying the relation
\[\beta^8 - 4\beta^7 + 12\beta^6 - 20\beta^5 + 24\beta^4 - 20\beta^3 + 12\beta^2 - 4\beta + 1=0.\]
We write
\begin{align*}
\beta'&:=\tfrac{1}{7}(\beta^7 + 2\beta^6 + 3\beta^5 + 5\beta^4 + 5\beta^3 + 3\beta^2 + 2\beta + 1)\\
\alpha&:=-9\beta' + 7\beta^6 - 9\beta^5 + 25\beta^4 - 14\beta^3 + 19\beta^2 - 4\beta + 2\\
\zeta&:=-11\beta' + 9\beta^6 - 12\beta^5 + 33\beta^4 - 18\beta^3 + 23\beta^2 - 5\beta +
    3\\
i&:=\zeta^2=\alpha^2-1\\
\epsilon&:=-8\beta' + 6\beta^6 - 7\beta^5 + 19\beta^4 - 7\beta^3 + 10\beta^2\\
\eta&:=\epsilon^3+\epsilon^2-\epsilon=\zeta\beta^2\\
\sqrt{2}&:=\epsilon^2-1,
\end{align*}
which fixes embeddings of $K_1=\QQ(\alpha)$ and $N_1=\QQ(\epsilon)$ into $M_2$. Note that $\Aut(M_2/\QQ)=D_4$, the dihedral group of order 8. We denote by $\sigma$ the involution of $M_2$ that leaves $N_1$ fixed and by $\tau$ the involution that leaves $K_1$ fixed. Then $\langle \sigma,\tau\rangle=\Aut(M_2/\QQ)$ and we write $\rho=(\sigma\tau)^2$, for the central involution of $\Aut(M_2/\QQ)$, which leaves $M_1$ fixed.
The unit groups of the rings of integers of these fields are $\calO_{K_1}^\times=\langle i,\alpha+1\rangle$, $\calO_{N_1}^\times=\langle -1,\epsilon, \eta\rangle$ and $\calO_{M_2}^\times=\langle \zeta,\alpha+1,\epsilon,\beta\rangle$.

We will find use for the following elementary lemmas which have undoubtedly been stated and proved many times, but for which we were unable to locate a reference.
\begin{lemma}\label{L:soroosh}
Let $p\equiv 1\pmod{8}$ be a prime and suppose that $x,y,D\in\ZZ$ with $p\nmid D$ and
\[x^2-Dy^2=p.\]
Then for $\alpha^2\equiv D\pmod{p}$ we have either $x+\alpha y\equiv 0\pmod{p}$ or
\[\legendre{\alpha(x+\alpha y)}{p}=1\]
\end{lemma}
\begin{proof}
We thank Soroosh Yazdani for pointing out the following proof.
Note that $x^2-Dy^2\equiv(x+\alpha y) (x-\alpha y)\equiv 0\pmod{p}$, so either $x+\alpha y\equiv 0\pmod{p}$ or $x-\alpha y\equiv 0\pmod{p}$. The lemma holds in the first case, so we assume the latter. Then
\[x(x+\alpha y)\equiv \alpha y(x+\alpha y)\pmod{p}\quad\text{ and }\quad2x\equiv (x+\alpha y)\pmod{p}.\]
It follows that
\[
\legendre{2\alpha y(x+\alpha y)}{p}=1.
\]
We are left with establishing that $\tleg{y}{p}=\tleg{2}{p}=1$. For any prime $q$ dividing $y$ we have $x^2\equiv p\pmod{q}$, so $\tleg{p}{q}=1$. Since $p\equiv 1 \pmod{8}$, quadratic reciprocity gives us $\tleg{q}{p}=1$ and $\tleg{-1}{p}=1$, so $y$ is a product of squares modulo $p$ and therefore a square modulo $p$ itself.
\end{proof}

\begin{lemma}\label{L:soroosh Z[i]}
Let $\pi \in \ZZ[i]$ be a prime element satisfying $\pi\equiv 1\pmod{2\ZZ[i]}$, $\Norm_{\ZZ[i]/\ZZ}(\pi)\equiv 1\pmod 8$ and $\tleg{1+i}{\pi}=1$. Suppose that $x,y,D\in\ZZ[i]$ with $\pi\nmid D$ and $x^2-Dy^2=\pi$. Then for $\alpha^2\equiv D\pmod{p}$ we have either $x+\alpha y\equiv 0\pmod{\pi}$ or
\[\legendre{\alpha(x+\alpha y)}{\pi}=1.\]
\end{lemma}

\begin{proof}
Note that quadratic reciprocity for $\ZZ[i]$ (established by Gauss and Dirichlet \cite{lemmermeyer:reciprocity}*{Proposition~5.1}) says that if $\pi,\lambda\in\ZZ[i]$ are distinct prime elements satisfying $\pi,\lambda\equiv 1\pmod{2\ZZ[i]}$ then
$\tleg{\lambda}{\pi}=\tleg{\pi}{\lambda}$. We can write
$y=i^a(1+i)^b\lambda_1^{e_1}\cdots\lambda_r^{e_2}$, where $\lambda_1,\ldots,\lambda_r\in\ZZ[i]$ are prime elements satisfying $\lambda_j\equiv 1\pmod{2\ZZ[i]}$ (we can ensure this by multiplying by $i$ is necessary).
It follows that $\tleg{\lambda_j}{\pi}=\tleg{\pi}{\lambda_j}$. The conditions in the lemma ensure that $\tleg{i}{\pi}=\tleg{1+i}{\pi}=1$.
This establishes the required ingredients to complete the proof in the same way as for
Lemma~\ref{L:soroosh}.
\end{proof}

\section{Class groups as local-global obstructions}

There are various ways to prove \eqref{E:V(2)} and \eqref{E:V(3)}. The proofs we give here are based on norm-form equations and are in the spirit of Gauss's treatment of genus theory. The lack of reference should not be construed as a claim to priority, but rather as evidence that it is hard to find a reference for such elementary facts.  These methods  have the great benefit that the techniques readily apply over extensions of the base ring as well. Doing so appears to provide a novel ingredient and allows us to prove something about $V(4)$. First we introduce some terminology and an elementary lemma that links solutions to norm form equations to divisibility in class groups.

Let $R$ be a principal ideal domain of characteristic different from $2$ and let $k$ be its field of fractions. Suppose $d\in R$ is a non-square. Let $L=k(\sqrt{d})$ and let $\calO_L\subset L$ be the integral closure of $R$ in $L$. We write $\Cl(\calO_L)$ for the ideal class group of $\calO_L$.

\begin{definition}
We say that a pair $(x,y)\in k\times k$ is \emph{$\sqrt{d}$-primitive} if the principal fractional ideal $(x+y\sqrt{d})\calO_L$ is integral and is not contained in the extension of any proper ideal from $R$ to $\calO_L$, i.e. for all $a\in R$ we have that 
\[(x+y\sqrt{d})\calO_L \subset a\calO_L \text{ if and only if }a\text{ is a unit in }R.\]
\end{definition}

The definition ensures that for a $\sqrt{d}$-primitive pair $(x,y)$, the principal ideal $(x+y\sqrt{d})\calO_L$ is not divisible by any prime ideals that are inert for $\calO_L/R$ and that if a split prime $\fq$ divides $(x+y\sqrt{d})\calO_L$ then the conjugate prime does not. This means that we can read off the exponents in the ideal factorization of $(x+y\sqrt{d})\calO_L$ from its norm $(x^2-dy^2)R$. Since $R$ is a principal ideal domain, this corresponds to the factorisation of $x^2-dy^2$ as an element of $R$.

\begin{remark}
If $\{1,\sqrt{d}\}$ forms an $R$-basis of $\calO_L$ then a pair $(x,y)$ is
$\sqrt{d}$-primitive if and only if $x,y\in R$ and $xR+yR=R$. In particular, if $R=\ZZ$ and $d$ is squarefree and $d \equiv 3 \pmod 4$ then $(x,y)$ is
$\sqrt{d}$-primitive if and only if $x,y\in\ZZ$ with $\gcd(x,y)=1$, which is the usual meaning of \emph{primitive}.
\end{remark}

\begin{lemma}\label{L:cglemma} Let $R$ be a principal ideal domain with field of fractions $k$. Let $L=k(\sqrt{d})$ be a quadratic extension of $k$ and let $\calO_L$ be the integral closure of $R$ in $L$.
Let $\fp\subset \calO_L$ be a prime ideal with norm $pR$. We have
\[[\fp]\in n\Cl(\calO_L)\]
if and only if there is a unit $u\in R^\times$ such that the equation
\[x^2-dy^2=up\,z^n\]
has a solution $x,y,z\in k$ with $(x,y)$ a $\sqrt{d}$-primitive pair.
\end{lemma}
\begin{proof} First suppose we have a solution with $(x,y)$ a $\sqrt{d}$-primitive pair. We denote the ideal factorisation of the principal ideal generated by $x+y\sqrt{d}$ by
\begin{equation}\label{E:ideal factorization}
(x+y\sqrt{d})\calO_L=\fp^{e_0}\prod_{i=1}^r \fq_i^{e_i}.
\end{equation}
The primitivity condition guarantees that none of the $\fp$ and $\fq_i$ are extensions of ideals in $R$, so each is either split or ramified. That means that $\Norm(\fq_i)=q_iR$ for some prime element $q_i$. Furthermore, note that if $\fq_i$ and $\fq_j$ have the same norm, then $\fq_i\fq_j=q_i\calO_L$, which would contradict the primitivity of $x,y$.

Taking norms of both sides of \eqref{E:ideal factorization} we obtain for some unit $u\in R^\times$ that
\[x^2-dy^2= up^{e_0}\prod_{i=1}^r q_i^{e_i}= upz^n.\]
Unique factorization gives that $e_0\equiv 1\pmod{n}$ and that $e_i\equiv 0\pmod n$ for $i=1,\ldots,r$. When we use that the left hand side of $\eqref{E:ideal factorization}$ is a principal ideal, we get the following identity in $\Cl(\calO_L)$:
\[0=[\fp^{e_0}\fq_1^{e_1}\cdots\fq_r^{e_r}]=[\fp]+[\fp^{e_0-1}\fq_1^{e_1}\cdots \fq_r^{e_r}]=[\fp]+n[\fa],\]
where $\fa=\fp^{(e_0-1)/n}\fq_1^{e_1/n}\cdots\fq_r^{e_r/n}$. This establishes one direction of the proof.

For the converse, let $\fa\subset \calO_L$ be an ideal such that $[\fp]=-n[\fa]$. This uses that $\calO_L$ is a Dedekind domain, so all ideal classes are represented by integral ideals. In fact, we can represent all ideal classes while avoiding a finite set of primes, so we can assume that $\Norm_{L/k}(\fa)$ is not divisible by $p$.
Note that inert ideals are principal and that conjugate primes represent inverse classes, so without loss of generality we have
$\fa=\fq_1^{e_1}\cdots \fq_r^{e_r}$, where the $\fq_i$ are split or ramified and have distinct norms. Then $\fp\fa^n$ is principal, so $\fp\fa^n=(x+y\sqrt{d})\calO_L$, where our assumptions on $\fa$ ensure that $(x,y)$ is a primitive pair. By picking $z\in R$ such that $zR=N(\fa)$, we obtain a solution as desired.
\end{proof}

\begin{proof}[Proof of \eqref{E:V(2)}]
We apply Lemma~\ref{L:cglemma} with $k=\QQ$, $L=H_0=\QQ(\sqrt{-p})$ and $\fp=\ft$ the ramified prime ideal of $\calO_L$ over $2$. We have already established that the class of $\ft$ has order $2$. We obtain that $p\in V(2)$ if and only if the equation
\begin{equation}\label{E:V(2) conic}
x^2+p y^2=2z^2
\end{equation}
has a solution such that $(x,y)$ is $(\sqrt{-p})$-primitive. A priori, we also need to consider the equation $x^2+p y^2=-2z^2$ but that obviously does not have primitive solutions.

Since Equation~\eqref{E:V(2) conic} is homogeneous, any non-zero solution is proportional to a $\sqrt{-p}$-primitive solution. Furthermore, the Hasse-Minkowski theorem guarantees that this equation has a solution if and only if it has solutions everywhere locally.

For solvability at $p$, one needs that $2$ is a square modulo $p$ and for solvability at $2$ one need that $p$ is a square modulo $4$. These conditions are met if and only if $p\equiv1 \pmod{8}$.
\end{proof}

\begin{proof}[Proof of \eqref{E:V(3)}]
Lemma~\ref{L:cglemma} gives $p\in V(3)$ if and only if there are $x,y,z\in\ZZ$ with $\gcd(x,y,z)=1$ such that
\[x^2+p y^2=2z^4.\]
We observe that this implies that $x,y$ are both odd and rewrite this to
\[-py^2=(x-\sqrt{2}z^2)(x+\sqrt{2}z^2).\]
Let $N_0=\QQ(\sqrt{2})$. We write $\tau$ for conjugation of $N_0/\QQ$, so for $\alpha=u+v\sqrt{2}$, we have $\act{\tau}{\alpha}=u-v\sqrt{2}$.

Since obviously $V(3)\subset V(2)$, we can assume that $p\equiv 1\pmod{8}$ by \eqref{E:V(2)}. Hence $p$ is split in $N_0$. Furthermore, since $\calO_{N_0}$ is a principal ideal domain and the fundamental unit $\epsilon=1+\sqrt{2}$ has norm $-1$, we have an element $\pi\in \calO_{N_0}$ such that $\pi\,\act{\tau}{\pi}=-p$. Primitivity implies that there is a $\gamma\in\calO_{N_0}$ with $\gamma\notin\sqrt{2}\calO_{N_0}$ such that
\[\begin{cases}
(x-z^2\sqrt{2})=\pm\pi\gamma^2\\
(x+z^2\sqrt{2})=\pm\act{\tau}{\pi}\,\act{\tau}\gamma^2.\\
\end{cases}\]
From this equation we derive that
\begin{equation}\label{E:V3conic}
\pm2\sqrt{2}z^2=\act{\tau}{\pi}\act{\tau}\gamma^2-\pi\gamma^2,
\end{equation}
Local solvability at $2$ forces the sign choice. Local solvability at $\pi\calO_L$ implies that
\begin{equation}
\legendre{\sqrt{2}\,\act{\tau}{\pi}}{\pi\calO_L}=1.
\end{equation}
Conversely, note that if we write $\gamma=s+t\sqrt{2}$ and collect coefficients with respect to $\sqrt{2}$ in \eqref{E:V3conic} then we get a conic in $s,t,z$
with solutions everywhere locally and hence globally. With some further standard calculations we can also check that we can find a point satisfying the appropriate primitivity conditions.

In order to simplify the symbol above, note that Lemma~\ref{L:soroosh} implies that
$\tleg{\sqrt{2}(1+\sqrt{2})\act{\tau}{\pi}}{\pi}=1$.
Furthermore, with the right choice of conjugates, one has $(1+\sqrt{2})(1+i)=(\zeta^3-1)^2$. Together this yields
\begin{equation}\label{E:1+sqrt2 vs 1+i}
\legendre{\sqrt{2}\,\act{\tau}{\pi}}{\pi\calO_L}=\legendre{1+\sqrt{2}}{p}=\legendre{1+i}{p},
\end{equation}
where the fact that $p\equiv 1\pmod{8}$ guarantees that the symbol is independent of choice of conjugate.
\end{proof}

Note that in the above two arguments, we obtained a criterion for $p\in V(e)$ for $e=2,3$ by reducing the condition in Lemma~\ref{L:cglemma} to the existence of a rational point on some conic, which is entirely determined by local conditions. We can handle the two cases above with $p$ as a parameter because the extensions involved in deriving the relevant conics are independent of $p$. For higher $e$ this does not seem to be the case anymore and this approach does not seem to have much benefit over computing $\Cl(\QQ(\sqrt{-p}))$ directly.

The following corollary to a classical result by Dirichlet (1842) allows us to link the class groups of $H_0=\QQ(\sqrt{-p})$ and $H_1=\QQ(\sqrt{p},i)$. We can then consider $H_1$ as a quadratic extension of $K_0=\QQ(i)$, whose ring of integers is a principal ideal domain. This allows us to apply Lemma~\ref{L:cglemma}, to obtain an alternative proof for \eqref{E:V(3)} and derive a new criterion for $p\in V(4)$.

\begin{proposition}[\cite{cohn:classinvit}*{Corollary 19.8c}]\label{P:cohn corr}
Let $p>0$ be a prime, let $h'=\#\Cl(\QQ(\sqrt{p}))$, let $h_0=\#\Cl(\QQ(\sqrt{-p}))$ and let $h_1=\#\Cl(\QQ(i,\sqrt{-p}))$. Then
\[
h_1=\begin{cases}
\frac{1}{2}h_0h'&\text{ if }p\equiv 1\pmod 4\\
h_0h'&\text{ if }p\equiv 3\pmod 4\text{ or }p=2
  \end{cases}
\]
\end{proposition}
For a prime satisfying $p\equiv 1\pmod{8}$ we have that $h'$ is odd by Gauss's genus theory.
Note that $2$ ramifies in $K_0=\QQ(i)$ and splits in $L_0=\QQ(\sqrt{p})$. That means that $H_1$ has two primes $\ft_1,\ft_2$ over $2$, each of ramification index $2$. Furthermore, since $h'$ is odd, we see that $[\ft_1^2]$ and $[\ft_2^2]$ have odd order in the class group.

Extension of ideals from $\calO_{H_0}$ to $\calO_{H_1}$ gives a homomorphism
\[\Cl(H_0)[2^\infty]\to \Cl(H_1)[2^\infty]\]
and it is easy to check that the kernel is of order $2$. In view of Proposition~\ref{P:cohn corr} this means that the map is surjective and thence that $\Cl(H_1)[2^\infty]$ is cyclic. The last fact also follows from applying genus theory to the relative extension $H_1/L_0$.

\begin{lemma}\label{L:cg iso criterion}
Let $p\equiv 1\pmod 8$ be a rational prime and let $e\geq 2$. We have $p\in V(e)$ if and only if the equation
\begin{equation}\label{E:K0conic}
x^2+p\,y^2=(1+i)z^{2^{e-2}}
 \end{equation}
has a solution $x,y,z\in\QQ(i)$ with
\[(x-i\,y)\ZZ[i]+2\,y\ZZ[i]=\ZZ[i].\]
\end{lemma}

\begin{proof}
Lemma~\ref{L:cglemma} with $(L,R,\fp,d,p)$ taken to be
$(H_1,\ZZ[i],\ft_1,p,1+i)$ (a shift in symbols used seems unavoidable here),
links the equation in the lemma to the question whether $[\ft_1]\in 2^{e-2}\Cl(H_1)$ and hence whether $p\in V(e)$. For $x\in\QQ(i)$ we write $\act{\sigma}{x}$ for its conjugate over $\QQ$. Note that if $(x,y)$ give rise to a solution then $(\act{\sigma}{x},\act{\sigma}{y}),(i\,x,i\,y),(\act{\sigma}{(i\,x)},\act{\sigma}{(i\,y}))$
give rise to solutions to $x^2+p\,y^2=u\alpha z^{2^{e-2}}$ where $u=i,-1,-i$. Therefore, the choice of the unit $u$ in Lemma~\ref{L:cglemma} does not affect solvability of the equation.

Note that $\{1,\frac{1}{2}(\sqrt{-p}+i)\}$ is a $\ZZ[i]$-basis of $\calO_{H_1}$. Therefore,
$(x,y)$ is a $\sqrt{-p}$-primitive pair in $\ZZ[i]$ if
\[x+y\sqrt{-p}=u+v\frac{\sqrt{-p}+i}{2},\]
with $u,v\in\ZZ[i]$ and $\gcd(u,v)=1$. That corresponds to the condition given in the lemma.
\end{proof}

\begin{proof}[Alternative proof of \eqref{E:V(3)}]
Since $V(3)\subset V(2)$, we can assume that $p\in V(2)$ and hence that $p\equiv 1\pmod 8$.
Lemma~\ref{L:cg iso criterion} yields that a necessary condition for $p\in V(3)$ is that
\[\legendre{i+1}{p}=1.\]
for both choices of $i$, because otherwise the conic given by \eqref{E:K0conic} does not even have local points at a place above $p$. However, note that
\[\legendre{1+i}{p}\legendre{1-i}{p}=\legendre{2}{p}=1\]
because $p\equiv 1\pmod 8$. Hence, the symbol does not depend on the choice of $i$. Furthermore, we can check that at $(1+i)\ZZ[i]$ there is no local obstruction to primitive solutions. The Hasse-Minkowski theorem once again guarantees the existence of rational solutions and the homogeneity of the equation allows us to derive primitive solutions from that. Therefore, the condition is also sufficient.
\end{proof}

\begin{proof}[Proof of Theorem~\ref{T:V(4)}]
Let us assume that $p\in V(3)$. By Lemma~\ref{L:cg iso criterion} we have that $p\in V(4)$ if and only if we have a solution $x,y,z\in\QQ(i)$ to
\[-p\,y^2=x^2-(1+i)z^4,\]
satisfying the additional conditions stated.
We adopt the notation from Section~\ref{S:notation} and factor this equation over $K_1$ to obtain
\[\begin{cases}
x+z^2\alpha=\delta\,\xi_1^2\\
x-z^2\alpha=\act{\rho}{\delta}\,\act{\rho}{\xi}_1^2
  \end{cases},
\]
for some $\delta$ representing a class in $K_1^\times/(K_1^{\times2})$ such that $N_{K_1/K_0}(\delta)\in -p K_0^{\times 2}$. Our primitivity condition together with the fact that $2$ is completely ramified in $K_1$ yields that $\delta$ can be represented by an algebraic integer that is a unit outside the primes above $p$.

Our conditions on $p$ ensure that $p$ is completely split in $K_1$. Let $u,v\in\ZZ$ be such that $p=u^2+v^2$ and suppose that $\pi_1,\ldots,\pi_4\in \calO_{K_1}$ such that $\Norm_{K_1/\QQ}(\pi_i)=p$ and
$\pi_1\pi_2=u+iv$ and $\pi_3\pi_4=u-iv$. The unit group of $\calO_{K_1}$ is generated by $\{i,1+\alpha\}$. Since $\Norm_{K_1/\QQ(i)}(1+\alpha)=-1$ is not a square,
the possible values for $\delta$ are
\[\pi_1\pi_3,\pi_1\pi_4,\pi_2\pi_3,\pi_2\pi_4,i\pi_1\pi_3,i\pi_1\pi_4,i\pi_2\pi_3,i\pi_2\pi_4\]

A necessary condition for $p\in V(4)$ is that
\begin{equation}\label{E:V4conic}
2\,z^2\alpha=\delta\,\xi_1^2-\act{\rho}{\delta}\,\act{\rho}{\xi}_1^2.
\end{equation}
has solutions everywhere locally. Noting that $i$ is a square modulo $p$, we see that there must be $j,k\in\{1,\ldots,4\}$ with $\{j,k\}\neq \{1,2\},\{3,4\}$ such that for all $l$ we have
\[\legendre{\pi_{j}\pi_{k}\alpha}{\pi_l}\neq -1.\]
However, note that Lemma~\ref{L:soroosh} yields identities such as
\[\legendre{\pi_1\pi_2}{\pi_3}=\legendre{u+iv}{\pi_3}=\legendre{i}{\pi_3}=1,\]
which allow us to deduce that the value does not depend on the actual choices of $j,k,l$ as long as $l\notin\{j,k\}$. Note that $\act{\rho}{\delta}\,\delta$ is a square locally at the prime above $2$, so we do not get any local obstructions there either. Therefore, the condition in the theorem is sufficient for \eqref{E:V4conic} to have points everywhere locally and hence globally. Checking that these points also give rise to primitive solutions is routine.
\end{proof}

\section{Congruent numbers: the first step}\label{S:congknown}

All classical results on congruent primes can be obtained via straightforward $2$-(isogeny) descent on either $E_p$ or one of its $2$-isogenous curves. See for instance \cite{hemenway:mscthesis}. We only state the parts that are important for our subsequent analysis.
\begin{enumerate}[(i)]
 \item If $p\equiv 3\pmod{8}$ then $\rk E_p(\QQ)=0$ and $\Sha(E_p/\QQ)[2]=0$.
 \item If $p\equiv 5,7 \pmod{8}$ then $\rk E_p(\QQ)\leq 1$. In fact, Monsky \cite{monsky:congnum} establishes equality and hence $\Sha(E_p/\QQ)[2]=0$.
 \item If $p\equiv 1\pmod{8}$ then $\rk E_p(\QQ)\leq 2$.
\end{enumerate}
In the last case, $p\equiv 1\pmod{8}$, some further work shows that the homogeneous spaces from \eqref{E:Cpi}
are everywhere locally solvable and that rational points on them would give rise to independent points on $E_p$. We analyze when this can be the case for $C_{p,1}$ and $C_{p,2}$.

\begin{lemma}\label{L:Dp1}
Let $p\equiv 1\pmod{8}$ be a prime. Then $C_{p,1}$ has a rational point if and only if the following curve has one.
\[D_{p,1}\colon v^2=p(u^4-4u^3-6u^2-12u-7)\]
Furthermore, $D_{p,1}$ has points everywhere locally if and only if $\tleg{1+i}{p}=1$.
\end{lemma}

\begin{proof}
For any rational point on $C_{p,1}$ there exists a value $\delta$ (determined up to squares) such that the point lifts to
\[D_{p,1}\colon
\begin{cases}
w^2=\delta(x^2+2x-1)\\
v^2=\frac{p}{\delta}(x^2-2x-1).
\end{cases}
\]
With some elementary resultant computation, one can show that it is sufficient to consider $\delta\in \{\pm 1,\pm 2, \pm p,\pm 2p\}$ and a straightforward local computation shows that for $\delta\in\{\pm2,\pm 2p\}$ the curve does not have $\QQ_2$-points.

Furthermore, the automorphisms of $C_{p,1}$ corresponding to $x\mapsto \frac{1}{x}$ and $x\mapsto -x$ show that the remaining values for $\delta$ all lead to isomorphic curves, so it is sufficient to consider $\delta=1$.

We parametrize the first conic by $(x,w)=\big(\frac{u^2+1}{2(u+1)},\frac{u^2+2u-1}{2(u+1)}\big)$. Substituting this parametrization into the second conic yields the given model of $D_{p,1}$. 

Note that since $p\equiv 1\pmod 8$, the local solvability of $D_{p,1}$ over $\QQ_2$ does not 
depend on $p$. Similarly, because $p>0$, the local solvability over $\RR$ does not depend on 
$p$ either. Given that $D_{p,1}$ has good reduction at all other primes, the only obstruction 
to local solvability can be at $p$. Note that if $(u_0,v_0)\in D_{p,1}(\QQ_p)$ then we need 
that $\ord_p(u_0^4-4u_0^3-6u_0^2-12u_0-7)$ is odd. For this we need that the quartic has a root 
in $\QQ_p$. Note that
\[u^4-4u^3-6u^2-12u-7=(u^2-2(1+\sqrt{2})u-1-\sqrt{2})(u^2-2(1-\sqrt{2})u-1+\sqrt{2})\]
and that the quadratics on the right hand side have discriminants $16(1\pm\sqrt{2})$. 
Furthermore, since $p\equiv 1\pmod 8$ we have $\sqrt{2}\in\QQ_p$, so $D_{p,1}(\QQ_p)$ is 
non-empty if and only if
\[\legendre{1+\sqrt{2}}{p}=\legendre{1+i}{p}=1;\]
see \eqref{E:1+sqrt2 vs 1+i} for the reason why the first equality holds.
\end{proof}

\begin{lemma}\label{L:Dp2}
Let $p\equiv 1\pmod{8}$ be a prime. Then $C_{p,2}$ has a rational point if and only if the following curve has one.
\[D_{p,2}\colon v^2=p(u^4-4u^3+24u+20)\]
Furthermore, the curve $D_{p,2}$ has points everywhere locally if and only if $\tleg{i+1}{p}=1$.
\end{lemma}

\begin{proof} A rational point on $C_{p,2}$ lifts, for some $\delta$, to
\[D_{p,2}\colon
\begin{cases}
w^2=\delta(x^2-2x+2)\\
v^2=\frac{p}{\delta}(x^2+2x+2).
\end{cases}
\]
The first conic only has real solutions if $\delta>0$ and
with an elementary resultant computation one can show it is sufficient to consider $\delta\in\{1,p,2,2p\}$. Furthermore, the automorphisms of $C_p$ corresponding to $x\mapsto \frac{1}{x}$ and $x\mapsto -x$ show that all choices lead to isomorphic curves, so it is sufficient to consider $\delta=1$.

We parametrize the first conic by $(x,w)=\big(\frac{2-u^2}{2u+2},\frac{u^2+2u+2}{2u+2}\big)$. Substitution into the second yields the given model of $D_{p,2}$.

For $p\equiv 1\pmod 8$, the curve $D_{p,2}$ has points at all primes outside $p$. For a point $(u_0,v_0)\in D_{p,2}(\QQ_p)$ we need $\ord_p(u_0^4-4u_0^3+24u_0+20)$ to be odd, so the quartic should have a root in $\QQ_p$. Note that
\[u^4-4u^3+24u+20=(u^2-(2-2i)u-4+2i)(u^2-(2+2i)u-4-2i)\]
and that the discriminants of the quadrics on the left hand side are
$(1+i)^9$ and $-i(1+i)^9$ respectively. The statement in the lemma follows by considering when these are squares in $\QQ_p$.
\end{proof}

Either of Lemmas~\ref{L:Dp1} and \ref{L:Dp2} establishes that primes $p\equiv 1\pmod{8}$ for which $\tleg{1+i}{p}=-1$ are \emph{not} congruent. This result is already mentioned in \cite{bastien:congnum} and \cite{tunnell:congnum}.
In order to interpret these results in terms of Question~\ref{Q:congnum}, we briefly review the relations between isogenies and Tate-Shafarevich groups in the next section.

\section{Sha and isogenies}

In this section, we review the conditions under which we can conclude the existence of $2^{e}$-torsion in $\Sha(E)$ by exhibiting $2^{e-1}$-torsion in $\Sha(E')$ for an appropriate $2$-isogenous elliptic curve $E'$. We use part of the proof that the truth of the Birch and Swinnerton-Dyer conjecture is constant in isogeny classes (see \cite{cassels:arith_g1-VIII} or \cite{milne:ADT}*{I.7}).

Let $E$ be an elliptic curve over a number field $k$ and let $\phi\colon E\to E'$ be an isogeny. Since elliptic curves are self-dual, we can interpret the isogeny dual to $\phi$ as $\phi^\vee\colon E'\to E$. Suppose that $m=\deg(\phi)$. Then the multiplication-by-$m$ homomorphism on $E$ factorizes as $m=\phi^\vee\circ \phi$.

Note that elements of $\Sha(E)$ are represented by principal homogeneous spaces $C$, so they have a free transitive $E$-action. We can use this together with an isogeny $\phi:E\to E'$ to induce a homomorphism
\[\begin{array}{cccc}
\phi\colon&\Sha(E)&\to&\Sha(E')\\
& C & \mapsto & C/\ker(\phi)
\end{array}\]
We write $\Sha(E)[\phi]$ for its kernel.

For any abelian group $A$ we define its \emph{divisible subgroup} to be
\[A_\div:=\{a\in A : a\in mA \text{ for all } m=1,2,\ldots\}\]
and write $A_\nd:=A/A_\div$. General results from descent show that the $p$-primary parts of $\Sha(E)_\nd$ are all finite.

The Cassels-Tate pairing yields non-degenerate, alternating pairings
\[\xymatrix{
\Sha(E)_\nd\ar[d]^\phi\ar@{}[r]|*+{\times} & \Sha(E)_\nd\ar[r]& \QQ/\ZZ\\
\Sha(E')_\nd \ar@{}[r]|*+{\times}& \Sha(E')_\nd\ar[u]^{\phi^\vee}\ar[r]& \QQ/\ZZ,
}\]
with the diagram commuting in the sense that
\[\langle \phi\xi,\xi'\rangle_{E'}=
\langle \xi,\phi^\vee\xi'\rangle_{E}.
 \]
In particular, the pairing induces an alternating, non-degenerate pairing
\begin{equation}\label{E:isosha}
\Sha(E')_\nd[\phi^\vee]\times\Sha(E')_\nd/\phi\Sha(E)_\nd\to \QQ/\ZZ.
\end{equation}

\begin{lemma}\label{L:isosha}
Let $\phi: E\to E'$ be a $p$-isogeny between elliptic curves over a number field $k$. Suppose that $\Sha(E')[\phi^\vee]=0$. 
Let $\xi\in\Sha(E)[p]$ and let $\xi'\in\Sha(E')[p]$ such that $\phi^\vee(\xi')=\xi$. Then
\[
\xi'\in p^{e-1}\Sha(E')\text{ implies that }\xi\in p^e\Sha(E).
\]
More generally, if $\Sha(E')$ has elements of order $p^e$ then $\Sha(E)$ has elements of order $p^{e+1}$.
\begin{proof}
Note that the first statement in the lemma is trivially true if $\xi'=0$ or more generally, if $\xi'$ is divisible.
Let us assume that $\xi''\in\Sha(E')$ such that $\xi'=p^{e-1}\xi''$.
From the pairing~\eqref{E:isosha} we see that $\Sha(E')[\phi^\vee]=0$ implies that $\Sha(E)\to\Sha(E')$ is surjective, so there is a $\xi'''\in \Sha(E)$ such that
$\phi(\xi''')=\xi''$. It follows that
\[\xi=\phi^\vee\circ p^{e-1}\circ \phi\, \xi'''=p^e\xi''',\]
which gives the statement in the lemma.

For the general observation, let $\xi''\in\Sha(E')$ be an element of order $p^e$, let $\xi'=p^{e-1}\xi''$ and let $\xi=\phi^\vee\xi'$. Then $\xi$ is an element of order $p$ that is divisible by $p^e$. This proves the statement.
\end{proof}

\end{lemma}

\section{Results from isogeny descents}\label{S:isodesc}
The curves $D_{p,i}$ arising in Lemmas~\ref{L:Dp1} and \ref{L:Dp2} are principal homogeneous spaces for the elliptic curves
\[\begin{aligned}
E_{p,1}\colon& y^2=x^3+4p^2x\\
E_{p,2}\colon& y^2=x^3+6p\,x^2+p^2x
  \end{aligned}\]
respectively. There are $2$-isogenies $\phi_i\colon E_p\to E_{p,i}$.

\begin{proof}[Proof of \eqref{E:W(2)}]
The following discussion holds for $i=1$ or $i=2$.
We do not give details here (see \cite{silverman:AEC}*{Proposition X.4.9}), but a $2$-isogeny descent on the pair $(E,E_{p,i})$ for a prime $p\equiv 1\pmod 8$ yields that $\rk E(\QQ)\leq 2$ as expected and that $\Sha(E_{p,i})[\phi_i^\vee]=0$ (the homogeneous spaces found there correspond to the $2$-torsion of $E$). It is perhaps worth noting that for the third $2$-isogenous curve $E_{p,3}: y^2=x^3-px^2+p^2x$, this is \emph{not} the case.

As established before, the curve $C_{p,i}$ represents a class $\xi_i$ in $\Sha(E)[2]$. Note that since $2=\phi_{p,i}^\vee \circ\phi_{p,i}$, divisibility of $\xi_i$ by $2$ implies that there is an everywhere locally solvable homogeneous space of $E_{p,i}$ that covers $C_{p,i}$. Lemma~\ref{L:Dp1} or \ref{L:Dp2} shows that this must be $D_{p,i}$, so if $\tleg{1+i}{p}=-1$ then $\xi_i$ is not divisible by $2$ and $\Sha(E_p)[2^\infty]=(\ZZ/2\ZZ)^2$ and $p\in W(1)\setminus W(2)$.

If $\tleg{1+i}{p}=1$ then the class $\xi'$ of $D_{p,i}$ in $\Sha(E_{p,i})$ satisfies $\phi_i^\vee(\xi')=\xi_i$, so $\xi_i$ is indeed divisible by $2$, so $p\in W(2)$.
\end{proof}

On the other hand, if $\tleg{1+i}{p}=1$ then Lemma~\ref{L:isosha} implies that $\xi_i$ is divisible by $2$. If we can find conditions on $p$ such that the class of $D_{p,i}$ is not divisible by $2$ in $\Sha(E_{p,i})$ then we classify when $\xi_i$ is divisible by $2$ but not by $4$ and hence when $p$ is not congruent and $\Sha(E_p)[2^\infty]=(\ZZ/4\ZZ)^2$.

The reason to concentrate on $D_{p,1}$ and $D_{p,2}$ is that we can write down nice quartic models for them without using any properties about $p$. Other homogeneous spaces would arise in a similar way, but the conic that requires parametrization in order to arrive at a quartic model may only be parametrizable if certain arithmetic conditions on $p$ are taken into account, i.e., that it is a prime $p\equiv 1\pmod{8}$. We will see in the next section how the nice models of $D_{p,1}$ and $D_{p,2}$ allow us to make one further step.

\section{Second descents}\label{S:second descent}

Let $p$ be a prime satisfying $p\equiv 1\pmod {8}$ and $\tleg{1+i}{p}=1$.
In this section, we perform a second descent on the curves $D_{p,1}$ and $D_{p,2}$ to determine if their classes are divisible by $2$ in $\Sha(E_{p,1})$ and $\Sha(E_{p,2})$. For any particular $p$, this is a completely standard procedure which can be executed automatically by several computer algebra systems. We give some details here because we do the calculation for an unspecified $p$, which is not completely automated.

We quickly review the parts of the method we have to refer to explicitly. The method we use is largely as suggested in \cite{mss:fourdesc}. However, we neglect to explicitly construct the coverings. See also \cite{brusto:twocov}.

We consider the smooth projective curve corresponding to the affine model
\[D\colon y^2= f(x)\]
where $f(x)$ is a square-free quartic polynomial over a field $k$ of characteristic $0$. If the leading coefficient of $f(x)$ is a square in $k$ then $D$ has $k$-rational points $P$ with $x(P)=\infty$. We denote these points with $\infty^+$ and $\infty^-$, where the $\pm$ superscript provides an arbitrary but fixed label.

The curve $D$ is a homogeneous space of an elliptic curve $E$ and invariant theory of binary quartic forms provides a degree $4$ map $D\to E$, providing $D$ with a torsor structure under $E[2]$ with base $E$. If $k$ is a number field and $D$ has points everywhere locally then $D$ represents a class in $\Sha(E)[2]$.

Let $L=k[\theta]=k[x]/(f(x))$. We consider the map
\[
\begin{array}{rccll}
\mu_k\colon&D(k)&\to&L^\times/L^{\times2}k^\times\\
&(x_0,y_0)&\mapsto&x_0-\theta&\text{if $y_0\neq 0$}\\
&(x_0,0)&\mapsto&(x_0-\theta)+\left.\frac{f(x)}{(x-x_0)}\right|_{x=\theta}\\
&P&\mapsto&1\text{ if } P=\infty^\pm.
\end{array}
\]
For $k=\QQ$ and a place $v$ of $k$ we write $L_v=L\otimes \QQ_v$ and consider the natural map $\rho_v: L^\times/L^{\times2}\QQ \to L_v^\times/L_v^{\times 2}\QQ_v^\times$. We define
\[\Selfake^{2}(D/\QQ)=\{c\in L^\times/L^{\times2}k^\times: \rho_v(c)\in\im\mu_{k_v}\text{ for all places }v\}.\]
We have that the class of $D$ is divisible by $2$ in $\Sha(E)$ if and only if $\Selfake^2(D/\QQ)$ is non-empty.

Let $S$ be the set of places of $k$ containing $2$, the places at infinity, the places where coefficients of $f$ are not integral and the places where $\disc(f)$ is not a unit. We write $\calO_{L,S}$ for the subring of $L$ of elements that are integral at all places $v$ outside $S$.

In our cases, $L$ is a number field where $\calO_{L,S}$ has class number $1$. In those cases, $\Selfake^2(D/\QQ)$ can be represented by elements in the finitely generated multiplicative group $\calO_{L,S}^\times$. Furthermore, for such elements $\delta$ we only have to check that $\rho_v(\delta)\in\im \mu_{k_v}$ for the places $v\in S$ in order to conclude that $\delta\in\Selfake^2(D/\QQ)$

Finally, if $f(x)$ has a root in $k_v$, then $\#\im \mu_{k_v}=\# \frac{E[2](k_v)}{|2|_v}$ and if $f(x)$ has no root in $k_v$ then $\#\im \mu_{k_v}=\frac{\#E[2](k_v)}{2|2|_v}$.

We use the notation for $K_0,K_1$ and their elements as introduced in Section~\ref{S:notation}.

\begin{proposition}\label{P:descent Dp1}
Let $p\equiv 1\pmod{8}$ be a prime satisfying $\tleg{1+i}{p}=1$. Let $\delta=\delta_p\in K_1=\QQ(\sqrt{1+i})$ be an algebraic integer such that $\Norm_{K_1/\QQ(i)}(\delta)=p$ and $\fp$ is a prime ideal above $p$ such that $\delta\notin\fp$. Then the following are equivalent.
\begin{enumerate}[(i)]
  \item The class of $D_{p,1}$ in $\Sha(E_{p,1})$ is divisible by $2$.
  \item The class of $D_{p,2}$ in $\Sha(E_{p,2})$ is divisible by $2$.
  \item $\displaystyle\legendre{\delta\zeta\sqrt{1+i}}{\fp}=1.$
\end{enumerate}
\end{proposition}

\begin{proof}[Proof that (i) is equivalent to (iii)]
We compute $\Selfake^{2}(D_{p,1})$ as sketched above. We have $f(x)=p(x^4-4x^3-6x^2-12x-7)$ and $L=N_1$. 

We can take $S=\{2,p,\infty\}$. Since $2$ and $p$ are completely ramified and split respectively, we have $4$ prime ideals $\fp_1,\ldots,\fp_4$ of $\calO_{N_1}$ above $p$. We choose generators for them in the following way. Let $\pi\in\calO_{M_2}$ be a generator of a prime ideal of $\calO_{M_2}$ above $p$.
Since $\calO_{M_2}^\times$ surjects onto $(\calO_{M_1}/2\calO_{M_1})^\times$, we can assume that $\pi\equiv 1\pmod {2\calO_{M_1}}$.
We define
\[\pi_1=\pi\,\act{\sigma}{\pi},\;
  \pi_2=\act{\rho}{\pi}\,\act{\sigma\rho}{\pi},\;
  \pi_3=\act{\tau}{\pi}\,\act{\sigma\tau}{\pi},\;
  \pi_4=\act{\tau\sigma}{\pi}\,\act{\sigma\tau\sigma}{\pi}\]
and write $\fp_i=\pi_i\calO_{N_1}$. Let $\iota: M_2\to \QQ_p$ be the completion corresponding to $\pi\calO_{M_2}$. We identify $M_2$ with its image under $\iota$. The completions $\iota_1,\ldots,\iota_4:N_1\to \QQ_p$ with respect to $\fp_i$ are induced by $\iota, \iota\rho, \iota\tau\sigma, \iota\sigma\tau$ respectively. 

From $\Norm_{N_1/\QQ}(\epsilon+1)=-2$ we know that
\[\calO_{L,S}^\times=\langle-1,\epsilon,\eta,\epsilon+1,\pi_1,\ldots,\pi_4\rangle.\]
Furthermore, $\pi_2\pi_3\pi_4=p/\pi_1$, so $\pi_1$ and $\pi_2\pi_3\pi_4$ represent the same class in $\calO_{L,S}^\times/\calO_{L,S}^{\times 2} \QQ^\times$. It is straightforward to check that classes in $\Selfake^2(D_{p,1})$ must be represented by $S$-integers that have norm in $p\QQ^{\times2}$. This means a full set of representatives for $\Selfake^2(D_{p,1})$ can be found in
\begin{equation}\label{E:Dp1 size8}
\{\pi_i,\eta\pi_i : i=1,\ldots,4\}.
\end{equation}
Note that $p$ is a square at $2$, so $\im \mu_{\QQ_2}$ is independent of $p$. We have that $f(x)$ is irreducible over $\QQ_2$ and that $\#E_{p,1}[2](\QQ_2)=2$, so $\#\im \mu_{\QQ_2} = 2$. We compute that
\[\ker (N\colon\frac{L_2^\times}{L_2^{\times2}\QQ_2^\times} \to \frac{\QQ_2^\times}{\QQ_2^{\times2}})\simeq(\ZZ/2\ZZ)^2,\]
that the images of $(0,\sqrt{-7p}),\infty^+\in D_{p,1}(\QQ_2)$ form $\im \mu_{\QQ_2}$. We find that it can be represented by $\{1,\epsilon^2\}+2\calO_{N_1}$.
Our earlier normalization ensures that $\pi_i\equiv 1\pmod{2\calO_{N_1}}$ and computation shows that $\eta\not\equiv\epsilon^2\pmod{2\calO_{N_1}}$, so from
\eqref{E:Dp1 size8} only $\{\pi_1,\ldots,\pi_4\}$ maps to $\im \mu_{\QQ_2}$.

Let $\theta=\epsilon^2-2\epsilon$ be a root of $f(x)$ in $N_1$. We know that $f(x)$ splits completely over $\QQ_p$. Let $\theta_1=\iota_1\theta,\ldots,\theta_4=\iota_4\theta$ be the roots of $f(x)$ in $\QQ_p$.
We fix $L\to L_p\simeq(\QQ_p)^4$ by $x\mapsto (\iota_1x,\ldots,\iota_4x)$.

We have $\#E_{p,1}[2](\QQ_p)=4$.  It is straightforward to check that $\im \mu_{\QQ_p}$ is represented by $\{(\theta_1,0),\ldots,(\theta_4,0)\}$, which in $L_p$ gives
\begin{equation}\label{E:Dp1 p-gens}
\begin{aligned}
(p(\theta_1-\theta_2)(\theta_1-\theta_3)(\theta_1-\theta_4),
   (\theta_1-\theta_2),
   (\theta_1-\theta_3),
   (\theta_1-\theta_4))\\
((\theta_2-\theta_1),
   p(\theta_2-\theta_1)(\theta_2-\theta_3)(\theta_2-\theta_4),
   (\theta_2-\theta_3),
   (\theta_2-\theta_4))\\
((\theta_3-\theta_1),
   (\theta_3-\theta_2),
   p(\theta_3-\theta_1)(\theta_3-\theta_2)(\theta_3-\theta_4),
   (\theta_3-\theta_4))\\
((\theta_4-\theta_1),
   (\theta_4-\theta_2),
   (\theta_4-\theta_3)),
   p(\theta_4-\theta_1)(\theta_4-\theta_2)(\theta_4-\theta_3)\\
\end{aligned}
\end{equation}
If $\pi_1$ represents a class in $\Selfake^{2}(D_{p,1})$ then based on valuations, $\pi_1$ and the first element listed in \eqref{E:Dp1 p-gens} must represent the same class modulo $L_p^{2\times}\QQ_p^\times$. That means that
\[
\legendre{\act{\rho}{\pi_1}\,\act{\tau\sigma}{\pi_1}\,\act{\sigma\tau}{\pi_1}(\theta_1-\theta_2)(\theta_1-\theta_3)(\theta_1-\theta_4)}{\pi}=
\legendre{\act{\rho}{\pi_1}(\theta_1-\theta_2)}{\pi}=
\legendre{\act{\tau\sigma}{\pi_1}(\theta_1-\theta_3)}{\pi}=
\legendre{\act{\sigma\tau}{\pi_1}(\theta_1-\theta_4)}{\pi}.
\]
We notice that the first two and the last two equalities lead to
\[\legendre{
 \act{\tau}{\pi}\act{\tau\sigma}{\pi}\act{\sigma\tau}{\pi}\act{\sigma\tau\sigma}{\pi}(\theta_1-\theta_3)(\theta_1-\theta_4)}{\pi}=
\legendre{\Norm_{M_2/N_0}(\act{\tau}{\pi})\sqrt{2}}{\pi}=1,
\]
which always holds by Lemma~\ref{L:soroosh}.
By equating the second and the last symbol we obtain
\[
\legendre{\act{\rho}{\pi}\act{\rho\sigma}{\pi}\act{\sigma\tau}{\pi}\act{\sigma\tau\sigma}{\pi}(\theta_1-\theta_2)(\theta_1-\theta_4)}{\pi}=
\legendre{\Norm_{M_2/K_1}(\act{\rho}{\pi}\act{\sigma\tau}{\pi})\zeta\alpha}{\pi}
\]
It is straightforward to check that $\delta=\Norm_{M_2/K_1}(\act{\rho}{\pi}\act{\sigma\tau}{\pi})$ satisfies the definition set out in the proposition, so this establishes that $\pi_1$ represents an element in $\Selfake^2(D_{p,1})$ if and only if $p$ satisfies condition \emph{(iii)}. The same conclusion holds for the other $\pi_i$ by symmetry.
\end{proof}

\begin{proof}[Proof that (ii) is equivalent to (iii)]
We follow the proof strategy we used for the first equivalence. We choose $\pi\in\calO_{M_2}$ in the same way and consider
\[\pi_1=\pi\,\act{\tau}{\pi},\;
  \pi_2=\act{\rho}{\pi}\,\act{\tau\rho}{\pi},\;
  \pi_3=\act{\sigma}{\pi}\,\act{\tau\sigma}{\pi},\;
  \pi_4=\act{\sigma\tau}{\pi}\,\act{\tau\sigma\tau}{\pi}.\]
The four prime ideals $\fp_1,\ldots,\fp_4$ of $\calO_{K_1}$ above $p$ are generated by the $\pi_i$ above and the completions $K_1\to\QQ_p$ with respect to $\fp_i$ are induced by the same embeddings $\iota,\rho\iota,\iota\tau\sigma,\iota\sigma\tau$ as before.

We have $f(x)=p(x^4-4x^3+24x+20)$ and $L=K_1$ and $\theta=\alpha^2-2\alpha$.
We can take $S=\{2,p,\infty\}$ and we have
\[\calO_{K_1,S}^\times=\langle i,\alpha+1,\alpha,\pi_1,\ldots,\pi_4 \rangle.\]
Note that $i=\alpha^4/2\in K_1^{\times 2}\QQ^\times$, so for representing $\Selfake^2(D_{p,2})$ we can ignore multiplication by $i$. Norm considerations show that a full set of representatives can be taken from the set
\[\{\pi_1,\ldots,\pi_4,(\alpha+1)\pi_1,\ldots,(\alpha+1)\pi_4\}.\]

A computation as before shows that the images of $\infty^+,(1,\sqrt{41p})\in D(\QQ_2)$  form $\im \mu_{\QQ_2}$, represented by $\{1,i\}+2\calO_{N_1}$.
Since $\alpha+1\not\equiv i \pmod{2\calO_{K_1}}$, we find that
$\Selfake^2(D_{p,2})$ can be represented by elements from $\{\pi_1,\ldots,\pi_4\}$.

By setting $\theta_j=\iota_j\theta$ we see that $\im\mu_{\QQ_p}$ is represented by the same formulas \eqref{E:Dp1 p-gens}. 
Based on valuations, $\pi_1$ can only represent the element corresponding to the first element there. For $\pi_1$ to represent the same class in $L_p^{\times2}/\QQ_p^\times$, we need that
\[
\legendre{\act{\rho}{\pi_1}\,\act{\tau\sigma}{\pi_1}\,\act{\sigma\tau}{\pi_1}(\theta_1-\theta_2)(\theta_1-\theta_3)(\theta_1-\theta_4)}{\pi}=
\legendre{\act{\rho}{\pi_1}(\theta_1-\theta_2)}{\pi}=
\legendre{\act{\tau\sigma}{\pi_1}(\theta_1-\theta_3)}{\pi}=
\legendre{\act{\sigma\tau}{\pi_1}(\theta_1-\theta_4)}{\pi}.
\]
The first and last equalities, together with  $(\theta_1-\theta_3)(\theta_1-\theta_4)$ being a square, yield
\[
\legendre{\act{\tau\sigma}{\pi}\act{\tau\sigma\tau}{\pi}\act{\sigma\tau}{\pi}\act{\sigma}{\pi}(\theta_1-\theta_3)(\theta_1-\theta_4)}{\pi}=\legendre{\Norm_{M_2/K_0}(\act{\sigma}{\pi})}{\pi}=1,
\]
which holds by Lemma~\ref{L:soroosh}. Equating the second and the fourth term yields
\begin{equation}\label{E:intermediate result}
\legendre{\act{\rho}{\pi}\act{\sigma\tau\sigma}{\pi}\act{\sigma}{\pi}\act{\sigma\tau}{\pi}\zeta\alpha}{\pi}=1
\end{equation}
From Lemma~\ref{L:soroosh Z[i]} with $D=i$ and $\alpha=\zeta$ we obtain
\[
\legendre{\act{\tau}{\pi}\act{\sigma\tau\sigma}{\pi}}{\pi}=\legendre{\Norm_{M_2/M_1}(\act{\tau}{\pi})}{\pi}=\legendre{\zeta}{\pi}.
\]
and another application of Lemma~\ref{L:soroosh} gives
\[
\legendre{\act{\sigma\tau}{\pi}\act{\tau}{\pi}\act{\tau\sigma}{\pi}\act{\sigma\tau\sigma}{\pi}}{\pi}=\legendre{\Norm_{M_2/N_0}(\act{\tau}{\pi})}{\pi}=
\legendre{\sqrt{2}}{\pi}=\legendre{\zeta}{\pi}.
\]
Multiplying these with \eqref{E:intermediate result} shows that it is equivalent to
\[
\legendre{\act{\rho}{\pi}\act{\sigma\tau\sigma}{\pi}\act{\sigma}{\pi}\act{\tau\sigma}{\pi}\zeta\alpha}{\pi}=1,
\]
which is the condition stated in the proposition. Conditions for the other $\pi_i$ follow by symmetry.
\end{proof}

\begin{proof}[Proof of Theorem~\ref{T:W(3)}]
We assume $p\in W(2)$. With Lemma~\ref{L:isosha} and the results of Section~\ref{S:isodesc} we have established that $[C_{p,i}]\in 4\Sha(E_p)$ if and only if $D_{p,i}\in 2\Sha(E_{p,i})$. This is exactly what either of the equivalences $\textit{(i)}\iff\textit(iii)$ and $\textit{(ii)}\iff \textit{(iii)}$ in Proposition~\ref{P:descent Dp1} establish.
\end{proof}

\section{Comparison of the methods}

This section gives an informal comparison of the methods of proof of
Theorems~\ref{T:V(4)} and \ref{T:W(3)}. They share some important characteristics and naturally the question arises to what extent a framework can be constructed in which both are applications of the same principle. A full answer is beyond the scope of this article (see for instance \cite{CT-Xu:BMforInt}) but we do sketch why Theorem~\ref{T:V(4)} is \emph{not} directly related to an isogeny.

Both questions can be interpreted in terms of local-global obstructions to rational or integral points on principal homogeneous spaces under algebraic groups.
The congruent number problem is directly formulated in this language.
We have the homogeneous spaces $C_{p,i}$ and $D_{p,i}$ under $E_p$ and $E_{p,i}$ respectively. In our case, for non-\emph{congruent} primes $p\equiv 1\pmod{8}$ and assuming that $\Sha(E_p)$ is finite, Lemma~\ref{L:isosha} yields that $\#\Sha(E_p)=4\#\Sha(E_{p,i})$, which means that analysis of the $2$- and $4$-torsion in $\#\Sha(E_{p,i})$ allows us to obtain information about $4$- and $8$-torsion in $\Sha(E_p)$. There is a long history of exploiting isogenies to obtain information about $\Sha(E)$ for elliptic curves $E$, see \cite{kramer:large_sha}.

For the class number problem, we consider integer points on $C\colon x^2+py^2=2$. The affine scheme described by this equation is a principal homogeneous space under the algebraic group scheme $T$ that describes the kernel of the norm map $\Norm\colon \QQ(\sqrt{-p})^\times\to\QQ^\times$. The fact that $C$ has integer points everywhere locally follows from the fact that there is an ideal $\ft$ of norm $2$ and the fact that $C$ does not have global integer points follows from the fact that $\ft$ is not a principal ideal.

The fundamental step that allows us to prove Theorem~\ref{T:V(4)} is to base extend to $\ZZ[i]$. Here $C$ does acquire an integer point and we are led to consider essentially the homogeneous space $D\colon x^2+py^2=1+i$ under $T'=T\times_\ZZ \ZZ[i]$ instead (noting that at $1+i$, the modified notion of primitivity actually describes a slightly different space than the model given here). Proposition~\ref{P:cohn corr} allows us to relate the information back to the quantities we are originally interested in.

Note that in the latter case, the two algebraic group schemes $T$ and $T'$ are \emph{not} isogenous. They are not even over the same base. For elliptic curves one can also use base extensions to kill part of $\Sha$, see for instance \cite{kramer:quad_ext}. This leads to particular instances of \emph{Mazur visibility} \cite{cremaz:vis}.  
See \cite{bruin:sha2vis} for an explicit description of these ideas regarding $\Sha(E)[2]$ for elliptic curves $E$. 

We searched for a proof of Theorem~\ref{T:W(3)} based on visibility but were unable to find an appropriate base extension or auxiliary elliptic curve that would work for all relevant $p$.

\section*{Acknowledgements}

We would like to thank Peter Stevenhagen for pointing out the similarities between the results for congruent numbers and those for imaginary quadratic class groups. Furthermore, we thank Soroosh Yazdani for useful discussions about Lemma~\ref{L:soroosh}, Tom Archibald for discussions about the history of the subject, Jeffrey Lagarias for useful comments and Alexa van der Waall for general discussions and suggestions.

\end{document}